\documentclass[a4paper,reqno,12pt]{amsart}
\usepackage{t1enc}
\usepackage[margin=1in]{geometry}
\usepackage{ifthen}
\usepackage{graphicx}

\newcommand{\R}{\mathbf{R}}

\newcommand{\pr}{\mathbf{P}}
\newcommand{\ex}{\mathbf{E}}
\newcommand{\N}{\mathbb{N}}

\newcommand{\la}{\lambda}

\usepackage{color}

\theoremstyle{plain}
\newtheorem{theorem}{Theorem}

\newtheorem{corollary}{Corollary}
\newtheorem{proposition}{Proposition}

\theoremstyle{definition}

\newtheorem{remark}{Remark}

\theoremstyle{remark}

\newcommand{\formula}[2][nolabel]
{\ifthenelse{\equal{#1}{nolabel}}
 {\begin{align*} #2 \end{align*}}
 {\ifthenelse{\equal{#1}{}}
  {\begin{align} #2 \end{align}}
  {\begin{align} \label{#1} #2 \end{align}}
 }
}

\numberwithin{equation}{section}

\begin{document}

%
%                            ---------- o ----------
%

\title [Yamada-Watanabe  theorem and  applications]{Multidimensional Yamada-Watanabe theorem\\ and its applications to particle systems}
\thanks{The authors were supported by MNiSW grant N~N201~373136 and the l'Agence Nationale de la Recherche grant ANR-09-BLAN-0084-01}
\author{Piotr Graczyk, Jacek Ma{\l}ecki}
\address{Piotr Graczyk \\ LAREMA \\ Universit\'e d'Angers \\ 2 Bd Lavoisier \\ 49045 Angers cedex 1, France}
\email{piotr.graczyk@univ-angers.fr}
\address{  Jacek Ma{\l}ecki,  \\ Institute of Mathematics and Computer Science \\ Wroc{\l}aw University of Technology \\ ul. Wybrze{\.z}e Wyspia{\'n}\-skiego 27 \\ 50-370 Wroc{\l}aw, Poland}
\email{jacek.malecki@pwr.wroc.pl }

\begin{abstract}
 A multidimensional    version of the Yamada-Watanabe theorem is proved. It implies
 a ``spectral'' matrix Yamada-Watanabe theorem.
It is also applied to particle systems of squared Bessel processes, corresponding 
  to matrix analogues of squared Bessel processes: Wishart and Jacobi matrix processes. The $\beta$-versions
  of these particle systems are also considered.
\end{abstract}

\maketitle
%
%                            ---------- o ----------
%
\section{Introduction}
\label{sec:intro}
In this article we  prove a multidimensional  analogue of the celebrated
Yamada-Watanabe theorem, ensuring the existence and {uniqueness} of strong solutions of one-dimensional stochastic differential equations (SDEs)
with  a  H\"older coefficient in the It\^o integral part.
It is proved in Section \ref{sec:yamada}, Theorem \ref{thm:betterYamada}.
 
 In Section \ref{sec:matrix}
 we derive a system of SDEs for the eigenvalues and the eigenvectors  
  for a solution of a matrix {SDE}
 of the form 
 \begin{equation*}
  dX_t=g(X_t) dB_t h(X_t) + h(X_t) dB_t^T g(X_t)  + b(X_t) dt
 \end{equation*}
 where $B_t$ is a Brownian matrix of dimension $p\times p$, the matrix stochastic process $X_t$
 takes values in the space of symmetric  $p\times p$ matrices and the functions $g,h,b:\R\rightarrow \R$
 act on the spectrum of   $X_t$. Under  some mild conditions on the functions  $g,h,b$ it is shown in
 Theorem \ref{prop:collision} that the eigenvalues never collide. The $\beta$-versions and complex versions of the 
 eigenvalue system are also considered for the collision time problem(Corollaries \ref{cor:beta},\ref{cor:complex}).
 
 If the functions $g,h,b$
 are such that $gh$ is 
 ${1/2}$-H\"older continuous, and   the
 symmetrized functions  $g^2 \otimes h^2$ and $b$ are
 Lipschitz continuous, then
 we establish in  Theorem \ref{th:yamadamatrix} the   existence and uniqueness of a strong solution 
 of the system of SDEs for the eigenvalues and the eigenvectors  of $X_t$. We call such a result
 ``a spectral matrix Yamada-Watanabe theorem''.

 Section \ref{sec:appli} contains interesting  applications.  We apply Theorems \ref{thm:yamada},  \ref{prop:collision} and  \ref{th:yamadamatrix}
 to

 (i) noncolliding particle systems of  squared Bessel processes which are intensely studied
in recent years in {statistical and mathematical physics} (Katori, Tanemura \cite{bib:katori,bib:katori2005,  bib:KatoriSugaku, bib:katori2011}).

(ii)  the systems of SDEs for the eigenvalues of Wishart
and Jacobi matrix processes, as well as of the $\beta$-Wishart and  $\beta$-Jacobi processes.
 We note   the importance of  the $\beta$-Wishart eigenvalues systems in statistical  physics: they are
 statistical mechanics models of ``$\log$-gases'', see the recent book of Forrester \cite{bib:forr}.
 
Surprisingly, the existence of strong solutions of SDEs for such H\"older-type  non-colliding particle systems  
was not established in general; only some   cases of (ii)  were treated by  Demni \cite{bib:demniCRAS,bib:demniARXIV,bib:demniJACOBI} and Lépingle \cite{bib:lep}.
In {Sections}  \ref{sec:systemsBessel} and  \ref{sec:systemsBeta} we prove the existence and uniqueness of a strong solution to  these systems of SDEs, for the whole range of the drift
parameter and $\beta\ge 1$.

The spectral matrix Yamada-Watanabe  theorem is   applied in Sections  \ref{sec:wishart}  and  \ref{sec:jacobi} to  matrix valued
 squared Bessel type processes:  Wishart and Jacobi matrix processes. We improve   the known results of Bru \cite{bib:b89, bib:bCRAS, bib:b91}, Mayerhofer et al. \cite{bib:mayer} and Doumerc \cite{bib:doumerc}, showing the existence and uniqueness of strong solutions of  the SDEs system for the eigenvalues and eigenvectors
of $X_t$,   for the whole range
of the drift parameter.

 In the Wishart case we contribute in this way to realization of a programme started
by Donati-Martin, Doumerc, Matsumoto and Yor \cite{bib:donatiyor}, claiming that Wishart processes have
  similar properties as classical $1$-dimensional  squared Bessel processes.

\section{A multidimensional Yamada-Watanabe theorem}
\label{sec:yamada}
Let us recall the classical Yamada-Watanabe theorem, see e.g.  \cite{bib:iw}, p.168 and \cite{bib:yw71}.
\begin{theorem}\label{thm:yamada}
Let $B(t)$ be a Brownian motion on $\R$. Consider the SDE
$$
dX(t)=\sigma(X(t))dB(t) + b(X(t))dt.
$$
If $|\sigma(x)-\sigma(y)|^2\le \rho(|x-y|)$  for a strictly increasing function $\rho$ on $\R^+$
with $\rho(0)=0$ and {$\int_{0^+}\rho^{-1}(x)dx=\infty$}, 
and $b$ is Lipschitz continuous, then the pathwise uniqueness of solutions
holds; consequently the equation has a unique strong solution.
\end{theorem} 
No multidimensional versions of the Yamada-Watanabe theorem seem to be known, even if their need is great (cf. \cite{bib:b91}, p. 738).  
 We propose a useful generalization, however we stress the fact that the H\"older continuous functions $\sigma_i$ appearing in the  following system of   SDEs are one-dimensional. The proof is based  on the  approach presented in Revuz,Yor \cite{bib:ry99}, in particular on the results of Le Gall \cite{bib:lg}.
 By $\|\cdot\|$ we mean the Euclidean norm $\|\cdot\|_2$ on $\R^d$.
\begin{theorem}
\label{thm:betterYamada}
  Let $p,q,r\in\N$ and the functions  $b_i:\R^{p}\to\R$, $i=1,\ldots,p$ and $c_k,d_j:\R^{p+r}\to \R$, $k=p+1,\ldots,p+q$, $j=p+1,\ldots,p+r$, be bounded real-valued and continuous,   satisfying the following Lipschitz conditions
  \formula{
     |b_i(y_1)-b_i(y_2)|&\leq A ||y_1-y_2||\/,\quad i=1,\ldots,p\/,\\
     |c_{k}(y_1,z_1)-c_k(y_2,z_2)|&\leq A||(y_1,z_1)-(y_2,z_2)|| \/,\quad k=p+1,\ldots,p+q\/,\\
     |d_j(y_1,z_1)-d_j(y_2,z_2)|&\leq A ||(y_1,z_1)-(y_2,z_2)|| \/,\quad j=p+1,\ldots,p+r\/,
  }
  for every $y_1,y_2\in\R^p$ and $z_1,z_2\in\R^r$. Moreover, let $\sigma_i:\R\to \R$, $i=1,\ldots,p$, be a set of bounded Borel functions such that
  \formula{
    |\sigma_i(x)-\sigma_i(y)|^2\leq \rho_i(|x-y|)\/,\quad x,y\in\R\/,
  }
  where $\rho_i:(0,\infty)\to (0,\infty)$ are Borel functions such that $\int_{0^+}\rho_i^{-1}(x)dx=\infty$.
 Then the pathwise uniqueness holds for the following system of stochastic differential equations
  \formula[eq:SDE:lemma]{
  dY_i &= \sigma_i({Y_i})dB_i+b_i(Y)dt\/,\quad i=1,\ldots,p\/,\\
  dZ_j &= \sum_{k=p+1}^{p+q}c_k(Y,Z)dB_k+d_j(Y,Z)dt\/,\quad j=p+1,\ldots,p+r\/,
  }
  where $B_1,\ldots,B_{p+q}$ are independent Brownian motions.
\end{theorem}
% \pg{  ({\bf REMARK: Signification of $p,q,r$}:  $p$=number of entries of $Y$  (later: of eigenvalues diagonal matrix $\Lambda$; $p$=number of eigenvalues)\\
%$r$=  number of entries of $Z$, (later: of eigenvector matrix $H$; $r=p^2$)\\
%$q$=number of Brownian Motions in SDE's for $dZ$ (later: number of independent BM's $\beta_{kj}$ in SDE's for $dH$; $q= p(p-1/2)$ .) 
%}

\begin{proof}
Let $(Y,Z)$ and $(\tilde{Y},\tilde{Z})$ be two solutions with respect to the same Brownian motion $B=(B_i)_{i\leq p+q}$ such that $Y(0)=\tilde{Y}(0)$ and $Z(0)=\tilde{Z}(0)$ a.s. For every $i=1,\ldots,p$ we have
\formula[eq:SDE:01]{
  Y_i(t)-\tilde{Y}_i(t) &= \int_0^t (\sigma_i(Y_i(s))-\sigma_i(\tilde{Y}_i(s)))dB_i(s) +\int_0^t (b_i(Y(s))-b_i(\tilde{Y}(s)))ds\/.
}
Then we get
\formula{
  \int_0^t \frac{\textbf{1}_{\{Y_i(s)>\tilde{Y}_i(s)\}}}{\rho_i(Y_i(s)-\tilde{Y}_i(s))}d\left<Y_i-\tilde{Y}_i,Y_i-\tilde{Y}_i\right> 
  =\int_0^t \frac{(\sigma_i(Y_i(s))-\sigma_i(\tilde{Y}_i(s)))^2}{\rho_i(Y_i(s)-\tilde{Y}_i(s))}\textbf{1}_{\{Y_i(s)>\tilde{Y}_i(s)\}}ds\leq t\/.
}

%\pg{ REMARK. Yes, the local time of 0-process at 0 is equal to 0:
%we take the limit of occupation times, where the integral is with respect to the
%quadratic variation of 0, so 0. This is the local time in the semi-martingale sense;
%in the local time for BM the integral is w.r. to the Lebesgue measure $dt$ }

Thus, applying Lemma 3.3 from \cite{bib:ry99}  p. 389 , we get that the local time of $Y_i-\tilde{Y}_i$ at $0$ vanishes identically. Consequently, by Tanaka's formula we get
\formula{
  |Y_i(t)-\tilde{Y}_i(t)| &= \int_0^t \textrm{sgn}(Y_i(s)-\tilde{Y}_i(s))d(Y_i(s)-\tilde{Y}_i(s))+L_t^0(Y_i-\tilde{Y}_i)\\
  &= \int_0^t \textrm{sgn}(Y_i(s)-\tilde{Y}_i(s))d(Y_i(s)-\tilde{Y}_i(s))\/.
}
Since $\sigma_i$ is bounded, using (\ref{eq:SDE:01}), we state that 
\formula{
 |Y_i(t)-\tilde{Y}_i(t)| - \int_0^t \textrm{sgn}(Y_i(s)-\tilde{Y}_i(s))(b_i(Y(s))-b_i(\tilde{Y}(s)))ds
}
is a martingale vanishing at $0$. This together with the Lipschitz conditions satisfied by $b_i$ give
\formula{
   \ex|Y_i(t)-\tilde{Y}_i(t)|\leq A\int_0^t \ex||Y(s)-\tilde{Y}(s)||ds\/.
}
Summing up the above-given inequalities we arrive at
\formula{
   \ex||Y(t)-\tilde{Y}(t)||\leq C\int_0^t \ex||Y(s)-\tilde{Y}(s)||ds
}
and Gronwall's lemma shows that $Y(t)=\tilde{Y}(t)$ for every $t>0$ a.s.

  Using in a  standard way the properties of the It\^o integral and the Schwarz inequality,
  similarly as in \cite{bib:iw}, p. 165,     we get that for every $t\in [0,T]$
\formula{
  \ex|Z_j(t)-\tilde{Z}_j(t)|^2 &\leq C\sum_{k=p+1}^{p+q}\ex(\int_0^t (c_k(Y(s),Z(s))-c_k(\tilde{Y}(s),\tilde{Z}(s))dB_k(s))^2\\
  &  +C\ex(\int_0^t (d_j(Y(s),Z(s))-d_j(\tilde{Y}(s),\tilde{Z}(s))ds)^2\\
  &\leq C\sum_{k=p+1}^{p+q}\ex\int_0^t (c_k(Y(s),Z(s))-c_k(\tilde{Y}(s),\tilde{Z}(s)))^2ds\\
  & +CT\ex\int_0^t (d_j(Y(s),Z(s))-d_j(\tilde{Y}(s),\tilde{Z}(s)))^2ds\\
  &\leq CA^2(q+T) \ex\int_0^t(||Y(s)-\tilde{Y}(s)||^2+||Z(s)-\tilde{Z}(s)||^2)ds
}
Thus, using the  previously proved fact that $Y=\tilde{Y}$ {a.s.} we get that
\formula{
  \ex||Z(t)-\tilde{Z}(t)||^2\leq CA^2(q+T)r\int_0^t \ex||Z(s)-\tilde{Z}(s)||^2ds\/.
}
 One more application of the Gronwall's lemma ends the proof.
\end{proof}

%\begin{remark}
% A typical example of functions $\sigma_i$ to which Theorem \ref{thm:betterYamada} applies is
% $\sigma_i(x)=\sqrt{x}$, with $\rho(x)=x$. m{Tu jest problem, bo w zalozeniach jest ograniczonosc $\sigma$}
%\end{remark}

%%%%%%%%%%%%%%%%%%%%%%%%%%%%%%%%%%%%%%%%%%%%%%%%%%%%%%%%%%%%%%%%%%%%%%%%%%%%%%%%%%%%%%%%%%%%%%%%%%%%%%%%%%%%%%%%%
%%%%%%%%%%%%%%%%%%%%%%%%%%%%%%%%%%%%%%%%%%%%%%%%%%%%%%%%%%%%%%%%%%%%%%%%%%%%%%%%%%%%%%%%%%%%%%%%%%%%%%%%%%%%%%%%%%
\section{Eigenvalues and eigenvectors of matrix stochastic processes}\label{sec:matrix}

\subsection{Real case}
Consider the space $\mathcal{S}_p$ of symmetric $p\times p$ real matrices. Recall that
if $g:\R\rightarrow \R$ and  $X\in\mathcal{S}_p$ then   $g(X)=Hg(\Lambda) H^T$,
where $X=H\Lambda H^T$
is a diagonalization of $X$, with $H$ an  orthonormal matrix  and $\Lambda$ a diagonal one. 
 Denote by  $B_t$  a Brownian $p\times p$ matrix.
 Let  $X_t$ be a stochastic process with values in  $\mathcal{S}_p$
 satisfying the matrix SDE
  \begin{equation}\label{eq:general}
  dX_t=g(X_t) dB_t h(X_t) + h(X_t) dB_t^T g(X_t)  + b(X_t) dt
 \end{equation}
 where $g,h,b:\R\rightarrow \R$, and  $X_0\in  \tilde {\mathcal S}_p$, the set of symmetric
 matrices with $p$ different eigenvalues.
 
 Let $\Lambda_t=diag(\lambda_i(t))$ be the diagonal matrix of eigenvalues of $X_t$
 ordered  increasingly: $\lambda_1(t)\le  \lambda_2(t) \le  \ldots  \le \lambda_p(t)$
 and $H_t$ an orthonormal  matrix of eigenvectors of $X_t$. Matrices
 $\Lambda$ and $H$   may be chosen(\cite{bib:norris}) as smooth functions of $X$ until the first
 collision time  
 $$\tau=\inf\{t:\ \lambda_i(t)=\lambda_j(t) {\rm\ for \ some\ }i\not=j\}.$$

We want to consider the SDEs
 satisfied by the processes of eigenvalues and eigenvectors of $X_t$.
In the sequel   we use the notation $dY d Z=d\langle Y,Z\rangle$ for the quadratic variation process.
Note that if $Y,Z$ are matrix valued processes, then $dYdZ$ is a matrix process (see e.g. \cite{PGLV}).
\begin{theorem}\label{th:eigen}
Suppose that an  $\mathcal{S}_p$-valued stochastic process $X_t$
satisfies the following matrix  stochastic differential equation
 \begin{equation*} 
  dX_t=g(X_t) dB_t h(X_t) + h(X_t) dB_t^T g(X_t)  + b(X_t) dt
 \end{equation*}
 where $g,h,b:\R\rightarrow \R$, and  $X_0\in  \tilde {\mathcal S}_p$.

Let $G(x,y)=g^2(x)h^2(y)+g^2(y)h^2(x)$. Then, for $t<\tau$ the eigenvalues process  $\Lambda_t$  and the eigenvectors process $H_t$
verify the following  stochastic differential equations

\begin{equation}\label{eq:dlambda}
 d\la_i=2g(\la_i)h(\la_i)d\nu_i +\left(b(\lambda_i)+\sum_{k\not=i} \frac{G(\lambda_i,\lambda_k)
}
 {\lambda_i-\lambda_k}\right) dt 
 \end{equation}

\begin{equation}\label{eq:dh}
dh_{ij}=\sum_{k\not=j}h_{ik} \frac{\sqrt{G(\la_j,\la_k)}}{\la_j-\la_k}d\beta_{kj}
 - \frac12 h_{ij} \sum_{k\not=j} \frac{G(\la_j,\la_k)}{(\la_k-\la_j)^2}dt
\end{equation}
where $(\nu_i)_i$ and $(\beta_{kj})_{k<j}$ are independent Brownian motions and $\beta_{jk}=\beta_{kj}$. 
\end{theorem}
\begin{proof}
The proof generalizes ideas of Bru \cite{bib:b89} in the case of Wishart processes. See also \cite{bib:katori2005}
for the SDEs for the eigenvalue processes of $X_t$.
Following \cite{bib:doumerc} in the case of matrix  Jacobi processes,  it is handy  to use the Stratonovich differential notation
$ X\circ dY= XdY +\frac12 dXdY$. We then write the It\^o product formula 
\formula{
d(XY)= dX\circ Y + X\circ dY.
}
We also have $dX\circ(YZ)=(dX\circ Y)\circ Z$ and $( X\circ dY)^T=dY^T \circ X^T$.\\
Define $A$,  a stochastic logarithm of $H${,} by
\formula{
dA=H^{-1}\circ dH= H^T\circ dH.
}
Observe that by It\^o formula applied to $H^TH=I$, the matrix $A$ is  skew-symmetric.
By It\^o formula applied to $\Lambda =H^T XH$, setting $dN=H^T\circ dX\circ H$, {we get}
$$
d\Lambda= dN +  \Lambda\circ dA - dA\circ \Lambda.
$$
The process   $\Lambda\circ dA - dA\circ \Lambda$  is zero on the diagonal.
Consequently $d\lambda_i=dN_{ii}$ and $0=dN_{ij}+(\lambda_i-\lambda_j)\circ dA_{ij}$, when $i\not= j$.
Thus
\begin{equation}\label{eq:daij}
 dA_{ij}=\frac1{\lambda_j-\lambda_i}\circ dN_{ij}, \ \ \ i\not=j.
\end{equation}
For further computations we need the quadratic variation
 $dX_{ij}dX_{km} $
which is easily computed from (\ref{eq:general}).  
\begin{equation*}
d X_{ ij} d X_{km}=\left( g^2(X)_{ik }h^2(X)_{jm } + g^2(X)_{im }h^2(X)_{jk } + g^2(X)_{jk }h^2(X)_{im } +   g^2(X)_{jm }h^2(X)_{ik }\right)dt.
\end{equation*}
The martingale  part of $dN$ equals the martingale part of $H^T dX\, H$ and by the last formula
\begin{equation}\label{eq:dndn}
dN_{ij}dN_{km}=\left( g^2(\Lambda)_{ik} h^2(\Lambda)_{jm} +
g^2(\Lambda)_{im } h^2(\Lambda)_{jk } +
g^2(\Lambda)_{jk } h^2(\Lambda)_{im }
+ g^2(\Lambda)_{jm } h^2(\Lambda)_{ik }\right)dt.
\end{equation}

From (\ref{eq:dndn}) it follows that 
\begin{equation}\label{eq:dniidnjj}
dN_{ii}dN_{jj}=4\delta_{ij}g^2(\la_i)h^2(\lambda_i)dt.
\end{equation}
Now  we compute the finite variation part  $dF$ of $dN$ 
 \begin{eqnarray*}
  dF&=&H^T b(X) Hdt + \frac12(dH^TdX\, H + H^TdXdH)\\
&=& b(\Lambda)dt + \frac12 \left( (dH^T\,H)(H^TdX\, H)+(H^TdX\, H) (H^TdH)\right)\\
&=& b(\Lambda)dt + \frac12(dNdA+(dNdA)^T).  
 \end{eqnarray*}
 
 Using (\ref{eq:daij}) and (\ref{eq:dndn}) we find, writing $G(x,y)=g^2(x)h^2(y)+g^2(y)h^2(x)$,
\formula{
 (dNdA)_{ij}=\sum_{k\not=j} dN_{ik}dA_{kj}= \delta_{ij} \sum_{k\not=i} \frac{G(\lambda_i,\lambda_k)
}{\lambda_i-\lambda_k} dt. 
 }
 It follows that the matrix $dNdA$ is diagonal, so also $dF$ is diagonal,
 \formula{
 dF_{ii}=b(\lambda_i)dt+\sum_{k\not=i} \frac{G(\lambda_i,\lambda_k)}{\lambda_i-\lambda_k} dt.
 }
 Finally, using (\ref{eq:dniidnjj}) and the last formula, there exist independent Brownian motions
 $\nu_i$, $i=1,\ldots,m$, such that  (\ref{eq:dlambda}) holds.
 
 In order to find SDEs for $H_t$, we deduce from the definition of $dA$
 that
  \formula{
  dH=H\circ dA=HdA +\frac12 dHdA=HdA+\frac12H dAdA{.}
}
By (\ref{eq:dndn}) we find $dN_{ij}dN_{ij}=g^2(\la_i)h^2(\la_j)+g^2(\la_j)h^2(\la_i)$
and $dN_{ij}dN_{km}=0$ when the ordered pairs $i<j$ and $k<m$ are different. We infer
 from (\ref{eq:daij}) that  
\begin{equation}\label{eq:daijexplicit}
dA_{ij}=\frac{\sqrt{G(\la_i,\la_j)}}{\la_j-\la_i}d\beta_{ij},
\end{equation}
where the Brownian motions $(\beta_{ij})_{i<j}$ are independent and {$\beta_{ji}=\beta_{ij}$}. Moreover, when $k<m$,
we have $d\lambda_i dA_{km}=dN_{ii}dN_{km}{/(\la_m-\la_k)}=0$ by (\ref{eq:dndn}),
so the Brownian motions $(\beta_{ij})_{i<j}$ and $(\nu_i)_i$ are independent.
From (\ref{eq:daijexplicit}) we deduce that the matrix $dAdA$ is diagonal and
\formula{
(dAdA)_{ii}=-\sum_{k\not=i}dA_{ik}dA_{ik}=-\sum_{k\not=i} \frac {G(\la_i,\la_k)}
{(\la_k-\la_i)^2}dt. 
}
Now we can compute $dH=HdA+\frac12H dAdA$ and prove (\ref{eq:dh}).
\end{proof}

\subsection{Complex case}
 In  this subsection we study the eigenvalues  process for a process  $X_t$
 with values in the space $\mathcal{H}_p$ of Hermitian $p\times p$ matrices.
 
\begin{theorem}\label{th:eigenC}
Let $W_t$ be a complex $p \times p$ Brownian matrix (i.e. $W_t=B^1_t+iB^2_t$ where $B^1_t$
and $B^2_t$ are two independent real Brownian $p\times p$ matrices).

Suppose that an  $\mathcal{H}_p$-valued stochastic process $X_t$
satisfies the following matrix  stochastic differential equation
 \begin{equation}\label{eq:generalC}
  dX_t=g(X_t) dW_t h(X_t) + h(X_t) dW_t^* g(X_t)  + b(X_t) dt {,}
 \end{equation}
 where $g,h,b:\R\rightarrow \R$, and  $X_0\in  \tilde {\mathcal H}_p$.

Let $G(x,y)=g^2(x)h^2(y)+g^2(y)h^2(x)$. Then, for $t<\tau$ the eigenvalues process  $\Lambda_t$ 
verifies the following system of  stochastic differential equations

\begin{equation}\label{eq:dlambdaC}
 d\la_i=2g(\la_i)h(\la_i)d\nu_i +\left(b(\lambda_i)+2\sum_{k\not=i} \frac{G(\lambda_i,\lambda_k)
}
 {\lambda_i-\lambda_k}\right) dt {,}
 \end{equation}
where $(\nu_i)_i$ are independent Brownian motions. 
\end{theorem}
\begin{proof}
We will need the following formula for the quadratic variation $ dX_{ij}dX_{kl} $
which is   computed from (\ref{eq:generalC}), using the fact that for a complex Brownian motion {$w_t$},
the quadratic variation processes satisfy { $dwdw=0$} and {$dwd\bar w  =2dt$}. 

\begin{equation}\label{eq:xxC}
d X_{ ij} d X_{kl}= 2\left( g^2(X)_{il}h^2(X)_{jk} + g^2(X)_{jk}h^2(X)_{il}\right)dt .
\end{equation}

Define $A$,  a stochastic logarithm of $H${,} by
$$
dA=H^{-1}\circ dH= H^*\circ dH.
$$
By It\^o formula applied to $H^* H=I$, the matrix $A$ is  skew-Hermitian.
In particular, the terms of $diag(A)$ 
are  purely imaginary (recall that in the real case they were 0).
By It\^o formula applied to $\Lambda =H^* XH$, we get,  setting $dN=H^*\circ dX\circ H$
\begin{equation*} 
d\Lambda= dN +  \Lambda\circ dA - dA\circ \Lambda.
\end{equation*}
We have
\formula{
dN= H^*dX H + \frac12(dH^*dX\, H + H^*dXdH)
}
so the process $N$ takes values in  $\mathcal{H}_p$. In particular its diagonal entries are real.
The process   $\Lambda\circ dA - dA\circ \Lambda$    is zero on the diagonal, so  $d\lambda_i=dN_{ii}$.
Moreover, when $i\not= j$, we have $0=dN_{ij}+(\lambda_i-\lambda_j)\circ dA_{ij}$ and
\begin{equation}\label{eq:daijC}
 dA_{ij}=\frac1{\lambda_j-\lambda_i}\circ dN_{ij}, \ \ \ i\not=j.
\end{equation}
The martingale  part of $dN$ equals the martingale part of $H^* dX\, H$ and by    formula (\ref{eq:xxC}) we obtain
\begin{equation}\label{eq:dndnC}
dN_{ij}dN_{km}=2\left(
g^2(\Lambda)_{im } h^2(\Lambda)_{jk } +
g^2(\Lambda)_{jk } h^2(\Lambda)_{im }\right)dt.
\end{equation}
From (\ref{eq:dndnC}) it follows that 
\begin{equation}\label{eq:dniidnjjC}
dN_{ii}dN_{jj}=4\delta_{ij}g^2(\la_i)h^2(\lambda_i)dt.
\end{equation}
Now  we compute the finite variation part  $dF$ of $dN$ 
 \begin{eqnarray*}
  dF &=& H^*b(X) Hdt + \frac12(dH^*dX\, H + H^*dXdH)\\
     &=& b(\Lambda)dt + \frac12 \left( (dH^*\,H)(H^*dX\, H)+(H^*dX\, H) (H^*dH)\right)\\
     &=& b(\Lambda)dt + \frac12(dNdA+(dNdA)^*).  
 \end{eqnarray*}
 
 Recall that $G(x,y)=g^2(x)h^2(y)+g^2(y)h^2(x)$.  We get
 \formula{
 (dNdA)_{ij}=\sum_k dN_{ik}dA_{kj}= 2\delta_{ij} \sum_{k\not=i} \frac{G(\lambda_i,\lambda_k)
}{\lambda_i-\lambda_k}dt + dN_{ij}dA_{jj}. 
 }
When $i=j$,  the term $dN_{ii}$ is real and $dA_{ii}\in i\R$. It follows that 
 \formula{
 dF_{ii}=b(\lambda_i)dt+2\sum_{k\not=i} \frac{G(\lambda_i,\lambda_k)}{\lambda_i-\lambda_k} dt.
 }
 Finally, using (\ref{eq:dniidnjjC}) and the last formula, there exist independent Brownian motions
 $\nu_i$, $i=1,\ldots,m$, such that  (\ref{eq:dlambdaC}) holds.
 \end{proof}
 %%%%%%%%%%%%%%%%%%%%%%%%%%%%%%%%%%%%%%%%%%%%%%%%%%%%%%%%%%%%%%%%%%%%%%%%%%%
The Theorem \ref{th:eigenC} may be applied in a special case $g(x)=\sqrt{x}$, $h(x)=1$ and $b(x)=2\delta>0$, when the equation (\ref{eq:generalC})
is the SDE for the complex Wishart process, called also a Laguerre process. This process
and its eigenvalues were studied   by K\"onig-O'Connell \cite{bib:konig} and Demni \cite{bib:demni}. 
%\pg {
%\begin{remark}
%\pg Surprisingly, the SDEs for the eigenvectors matrix $H_t$
% remain an open problem in the complex Hermitian case. As in the real case, we have
%  $$
%  dH=H\circ dA=HdA +\frac12 dHdA=HdA+\frac12H dAdA
%$$
%but this equation contains diagonal entries $dA_{jj}$ with purely imaginary values.
% By (\ref{eq:daijC}), SDEs for  $A_{ij}$ are easy to write, as in the real case.
% However, no information about SDEs verified by  $A_{jj}$  is available.
%\end{remark}
%}

\subsection{Collision time}
In this subsection we show that under some mild conditions on the functions $g,h$ and $b$ in the matrix  SDE  (\ref{eq:general}),
the eigenvalues of the process $X_t$ never collide.
\begin{theorem}\label{prop:collision}
Let $\Lambda=(\lambda_i)_{i=1\ldots p}$ be a process starting from $\lambda_1(0)<\ldots<\lambda_p(0)$ and satisfying (\ref{eq:dlambda}) with functions 
$b,g, h:\R\to \R$ such that $b,g^2,h^2$ are Lipschitz continuous   and $g^2h^2$ is convex or in class  $\mathcal{C}^{1,1}$.  Then the first collision time $\tau$ is infinite a.s.
\end{theorem}
\begin{proof}
We define $U = -\sum_{i<j}\log(\lambda_j-\lambda_i)$ on $t\in[0,\tau]$. Applying It\^o formula, using (\ref{eq:dlambda}) and the fact that $d \lambda_id\lambda_j =\delta_{ij}4g^2(\lambda_i)h^2(\lambda_i)dt $ we obtain 
\formula{
  dU &= \sum_{i<j}\frac{d\lambda_i-d\lambda_j}{\lambda_j-\lambda_i}+\frac12\frac{d\left<\lambda_i,\lambda_i\right>+d\left<\lambda_j,\lambda_j\right>}{(\lambda_j-\lambda_i)^2} = dM + dA^{(1)}+dA^{(2)}+dA^{(3)}\/,
}
where
\formula{
   dM &= 2\sum_{i<j}\frac{g(\lambda_i)h(\lambda_i)d\nu_i-g(\lambda_j)h(\lambda_j)d\nu_j}{\lambda_j-\lambda_i}{,}\\
   dA^{(1)} & = \sum_{i<j}\frac{b(\lambda_i)-b(\lambda_j)}{\lambda_j-\lambda_i}\,dt{,}\\
   dA^{(2)} & = 2\sum_{i<j}\frac{(g^2(\lambda_j)-g^2(\lambda_i))(h^2(\lambda_j)-h^2(\lambda_i))}{(\lambda_j-\lambda_i)^2}\,dt{,} \\
   dA^{(3)} & = \sum_{i<j}\frac{1}{\lambda_j-\lambda_i}\sum_{k\neq i, k\neq j}\left(\frac{G(\lambda_i,\lambda_k)}{\lambda_i-\lambda_k}-\frac{G(\lambda_j,\lambda_k)}{\lambda_j-\lambda_k}\right)dt\\
   & = \sum_{i<j<k}\frac{G(\lambda_j,\lambda_k)(\lambda_k-\lambda_j)-G(\lambda_i,\lambda_k)(\lambda_k-\lambda_i)+G(\lambda_i,\lambda_j)(\lambda_j-\lambda_i)}{(\lambda_j-\lambda_i)(\lambda_k-\lambda_i)(\lambda_k-\lambda_j)}dt{.}
}
We will show that the finite variation part of $U$ is bounded on any interval $[0,t]$.
Lipschitz continuity of $b$, $g^2$ and $h^2$ implies that $|A^{(1)}_t|\leq K{p(p-1)}t/2$ and 
$|A^{(2)}_t|\leq K^2{p(p-1)}t$, where $K$ is a constant appearing in the Lipschitz condition. Observe also that if for every $x,y,z$ we set
\formula{
   H(x,y,z) = [(g^2(x)-g^2(z))(h^2(y)-h^2(z))+(g^2(y)-g^2(z))(h^2(x)-h^2(z))](y-x){,}
}
then $H(x,y,z) = (G(x,y)-G(x,z)-G(y,z)+G(z,z))(y-x)$ and 
\formula{
 &H(x,y,z)+H(y,z,x)-H(x,z,y) = 2(z-y)G(y,z)-2(z-x)G(x,z)\\
 &+2(y-x)G(x,y) +G(x,x)(z-y)-G(y,y)(z-x)+G(z,z)(y-x).
}
Using the last equality and the fact that $|H(x,y,z)|\leq 2K^2|(y-x)(z-y)(z-x)|$ we can write $2dA^{(3)}= dA^{(4)}+dA^{(5)}$, where $0\le A^{(4)}_t\leq K^2p(p-1)(p-2)t/6$ and
\formula{
  dA^{(5)}_t &= \sum_{i<j<k}\frac{G(\lambda_j,\lambda_j)(\lambda_k-\lambda_i)-G(\lambda_i,\lambda_i)(\lambda_k-\lambda_j)-G(\lambda_k,\lambda_k)(\lambda_j-\lambda_i)}{(\lambda_j-\lambda_i)(\lambda_k-\lambda_i)(\lambda_k-\lambda_j)}dt\\
  &= \sum_{i<j<k}\left(\frac{G(\lambda_j,\lambda_j)-G(\lambda_i,\lambda_i)}{\lambda_j-\lambda_i}-\frac{G(\lambda_k,\lambda_k)-G(\lambda_j,\lambda_j)}{\lambda_k-\lambda_j}\right)\frac{1}{\lambda_k-\lambda_i}dt
}
If $G(x,x)=2g^2(x)h^2(x)$ is convex then obviously the expression under the last sum  and $A^{(5)}$ is non-positive. When $G(x,x)$ is $\mathcal{C}^{1,1}$, (i.e. $|G'(x,x)-G'(y,y)|\leq C|x-y|$) then it is bounded by $C$ and  $|A^{(5)}_t|\leq C t$.   

Since finite-variation part of $U$ is finite whenever $t$ is bounded, applying McKean argument
(see  \cite{bib:McKean, bib:mayer}) 
we obtain that $U$ can not explode in finite time with positive probability and consequently $\tau=\infty$ a.s. 
\end{proof}
\begin{remark}
 Note that if $p= 2$ then the assumptions on $g^2h^2$ can be dropped since in that case $dA^{(3)}\equiv 0$.   
\end{remark}

 In the modern theory of particle systems it is important to consider and to study $\beta$-versions of a particle system
 given by the SDEs system  (\ref{eq:dlambda}), i.e. the solutions of the SDEs system
 \begin{equation}\label{eq:dlambdaBETA}
 d\la_i=2g(\la_i)h(\la_i)d\nu_i +\beta  \left(b(\lambda_i)+\sum_{k\not=i} \frac{G(\lambda_i,\lambda_k)
}
 {\lambda_i-\lambda_k}\right) dt,\ \ \ \beta>0. 
 \end{equation}
 Note that the system (\ref{eq:dlambdaBETA}) is for $\beta\not=1$  no longer of the form (\ref{eq:dlambda}), because
 $\beta G(x,y)\not= g^2(x)h^2(y)+ g^2(y)h^2(x)$. However, we have
 \begin{corollary}\label{cor:beta}
  Let $\Lambda=(\lambda_i)_{i=1\ldots p}$ be a process starting from $\lambda_1(0)<\ldots<\lambda_p(0)$ and satisfying (\ref{eq:dlambdaBETA}) with functions 
$b,g, h:\R\to \R$ such that $b,g^2,h^2$ are Lipschitz continuous   and $g^2h^2$ is convex or in class  $\mathcal{C}^{1,1}$. If $\beta \ge 1$ then the first collision time $\tau$ is infinite a.s.
 \end{corollary}

\begin{proof}
  The proof is similar to the proof of the Theorem \ref{prop:collision}, with the decomposition $dU=dM + dA^{(1)}+dA^{(2)}+dA^{(3)}$
 given by
\formula{
   dM &= 2\sum_{i<j}\frac{g(\lambda_i)h(\lambda_i)d\nu_i-g(\lambda_j)h(\lambda_j)d\nu_j}{\lambda_j-\lambda_i}{,}\\
   dA^{(1)} & =\beta \sum_{i<j}\frac{b(\lambda_i)-b(\lambda_j)}{\lambda_j-\lambda_i}\,dt{,}\\
   dA^{(2)} & = 2\sum_{i<j}\frac{(g^2(\lambda_j)-g^2(\lambda_i))(h^2(\lambda_j)-h^2(\lambda_i))}{(\lambda_j-\lambda_i)^2}\,dt +
   2(1-\beta)\sum_{i<j}\frac{G(\lambda_i,\lambda_j)}{(\lambda_j-\lambda_i)^2},\\
   dA^{(3)} & = \beta\sum_{i<j}\frac{1}{\lambda_j-\lambda_i}\sum_{k\neq i, k\neq j}\left(\frac{G(\lambda_i,\lambda_k)}{\lambda_i-\lambda_k}-\frac{G(\lambda_j,\lambda_k)}{\lambda_j-\lambda_k}\right)dt.}
   The estimates of $M, A^{(1)}$, $A^{(3)}$ and of the first term of $A^{(2)}$ are identical as in the proof of Theorem \ref{prop:collision}. The term 
 $2(1-\beta)\sum_{i<j}\frac{G(\lambda_i,\lambda_j)}{(\lambda_j-\lambda_i)^2}$ is less or equal 0 for $\beta\ge 1$, so $A^{(2)}$ cannot explode
 to $+\infty$ and neither can $U$. 
 \end{proof}
 \begin{remark}
 The condition $\beta\ge 1$ is optimal in Corollary \ref{cor:beta}. It is known (\cite{bib:RShi,bib:KatoriSugaku}) that the Dyson Brownian motion
 defined as a solution of the SDEs system
 $$
 dY_i=d\nu_i + \beta\sum_{k\not=i}\frac{1/2}{Y_i-Y_k} dt, \ \ \ i=1,\ldots,p,
 $$
 has collisions for $\beta<1$. Note that taking $g(x)=1/2, h(x)=1, b(x)=0$ and $\beta=1$ the Dyson SDEs system is
 of the form  (\ref{eq:dlambda}) and  a general Dyson SDEs system is the $\beta$-version of the $\beta=1$ case. 
 \end{remark}
  Observe  that the  Theorem \ref{prop:collision} holds also  in the complex case. 
 \begin{corollary}\label{cor:complex}
  Under the hypotheses of Theorem \ref{prop:collision}, the solutions of the SDEs system (\ref{eq:dlambdaC}), i.e.
  the eigenvalues
 of the process $X_t$ on $\mathcal{H}_p$,
  verify $\tau=\infty$ a.s. It is also true for their $\beta$-versions with $\beta\ge \frac12$.
  \end{corollary}
\begin{proof}
 Note that the system (\ref{eq:dlambdaC}) is the $\beta=2$-version of the system (\ref{eq:dlambda}) with the same $g$ and $h$
 and $b/2$ instead of $b$. 
\end{proof}

\subsection{Spectral matrix Yamada-Watanabe theorem}

\begin{theorem}\label{th:yamadamatrix}
 Consider the matrix  SDE (\ref{eq:general}) on $\mathcal{S}_p$
 \begin{equation*}
  dX_t=g(X_t) dB_t h(X_t) + h(X_t) dB_t^T g(X_t)  + b(X_t) dt
 \end{equation*}
 where $g,h,b:\R\rightarrow \R$ and $X_0\in  \tilde {\mathcal S}_p$.
 Suppose that 
 \begin{equation}\label{ineq:rho}
    | g(x)h(x)- g(y)h(y) )|^2\leq \rho(|x-y|)\/,\quad x,y\in\R\/,
 \end{equation} 
  where $\rho:(0,\infty)\to (0,\infty)$ is a  Borel function such that $\int_{0^+}\rho^{-1}(x)dx=\infty$,
   that the function $G(x,y)=g^2(x)h^2(y)+g^2(y)h^2(x)$ is locally Lipschitz
  and strictly positive on $\{x\not=y\}$ and that $b$ is   locally Lipschitz.
  Then, for $t<\tau$, the   pathwise uniqueness holds for the eigenvalue and eigenvector processes of $X_t$, solutions of the  SDEs
  system (\ref{eq:dlambda}) and  (\ref{eq:dh}).
  \end{theorem}
\begin{remark}
 The hypothesis {in Theorem \ref{th:yamadamatrix}} on the strict positivity of $G(x,y)$ off the diagonal $\{x=y\}$
 is equivalent to the condition that $g$ and $h$ have not more than one zero and
 their zeros  are not common.
\end{remark}

\begin{remark}
 In the matrix  SDE (\ref{eq:general}) the functions $g$ and $h$ appear  only in the martingale part, whereas
 in the equations (\ref{eq:dlambda}) and (\ref{eq:dh}) they intervene also in the finite variation part.
 That {is} why a Lipschitz condition  on the symmetrized function $g^2\oplus h^2$ cannot be avoided in a   spectral matrix
 Yamada-Watanabe theorem on $\mathcal{S}_p$.
\end{remark}
\begin{proof}
 We diagonalize $X_0=h_0\lambda_0 h_0^T$.   We will show that 
  the equations (\ref{eq:dlambda}) and (\ref{eq:dh}) have unique strong solutions when $\Lambda_0=\lambda_0$ and $H_0=h_0$. 
The functions
   \formula{
   b_{i}(\lambda_1,\ldots,\lambda_p) &= b(\la_i)+\sum_{k\neq i}\frac{G(\la_i,\la_k)}{\lambda_i-\lambda_k}{,} \\ 
   c_{ij}(\lambda_1,\ldots,\lambda_p,h_{11},h_{12},\ldots,h_{pp}) &= \delta_{kj}h_{ik} \frac{\sqrt{G(\la_j,\la_k)}}{\la_j-\la_k} {,}
  \\
   d_{ij}(\lambda_1,\ldots,\lambda_p,h_{11},h_{12},\ldots,h_{pp}) &= - \frac12 h_{ij} \sum_{k\not=j} \frac{G(\la_j,\la_k)}{(\la_k-\la_j)^2} 
   }
   are locally Lipschitz continuous on $D=\{0\leq \lambda_1< \lambda_{2}<\ldots<\lambda_p\}\times [-1,1]^r$, $r=p^2$. Thus, they can be extended from the compact sets 
   \formula{
      D_m = \{0\leq \lambda_1< \lambda_{2}<\ldots<\lambda_p<m, \lambda_{i+1}- \lambda_i\geq 1/m\}\times [-1,1]^r
   }
   to bounded Lipschitz continuous functions on $\R^{p+r}$. We will denote by $b_{i}^m$, $c_{ik}^m$ and $d_{ij}^m$ such extensions for $m=1,2,\ldots$. 
   
   We consider the following system of SDE (recall that $\beta_{kj}=\beta_{jk}$)
   \formula{
 d\lambda_i^m &= 2 g(\lambda_{i}^m)h(\lambda_{i}^m) d\nu_i +b_i^m(\Lambda^m)dt\/,\quad  i=1,\ldots,p\/,\\ 
     dh_{ij} &= \sum_{k\not=j} c_{ij}^m(\Lambda^m,H)\,d\beta_{kj}(t)+d_{ij}^m(\Lambda^m,H)\,dt\/,\ \ \ \  1\leq i,j\leq p\/.
 }
  Since  $| g(x)h(x)- g(y)h(y) )|^2\leq \rho(|x-y|)$  
  and $\int_{0^+}{\rho(x)^{-1}dx}=\infty$, by  Theorem~\ref {thm:betterYamada} with $q=\frac12p(p-1)$, we obtain that there exists a unique strong solution of the above-given system of SDEs. Using the fact that $D_m\subset D_{m+1}$, $\lim_{m\to \infty} D_m = D$ and the standard procedure we get that there exists a unique strong solution
  $(\Lambda_t,H_t)$ of the systems (\ref{eq:dlambda}) and (\ref{eq:dh}) up to the first exit time  from the set $D$. This time is almost surely
  equal to $\tau$, the first collision time of the eigenvalues.
\end{proof}
  The Theorems \ref{th:yamadamatrix} and \ref{prop:collision}  imply the following global strong existence result for eigenvalues and eigenvectors  of a matrix SDE on
  the space  $\mathcal{S}_p$.
  \begin{corollary}\label{cor:global}
   Suppose that  $b,g^2,h^2$ are Lipschitz continuous,  $g^2h^2$ is convex or in class  $\mathcal{C}^{1,1}$ 
   and that $G(x,y)$ is strictly positive on $\{x\not=y\}$. Then the  system of  SDEs  (\ref{eq:dlambda}) and  (\ref{eq:dh}) for eigenvalue and eigenvector processes  of the matrix process  on $\mathcal{S}_p$   given by (\ref{eq:general})
   admits a unique strong solution on $[0,\infty)$.
  \end{corollary}
  \begin{proof}
  Recall that if a non-negative function $F$ is Lipschitz continuous then $\sqrt{F}$ is {$1/2$}-H\"older continuous.
 Observe that if $g^2$ and $h^2$ are Lipschitz continuous then $g^2h^2$ is locally Lipschitz and  $gh$ is {$1/2$}-H\"older. Thus  (\ref{ineq:rho})
 is verified and  the Theorem \ref{th:yamadamatrix} applies on $[0,\tau)$. By  Theorem \ref{prop:collision}, $\tau=\infty$
 almost surely.
  \end{proof}
  \begin{remark}\label{dimensions}
   Theorem \ref{th:yamadamatrix} and Corollary \ref{cor:global} establish the pathwise uniqueness and strong existence of   eigenvalue and eigenvector processes $\Lambda_t$
   and $H_t$ 
   of the process $X_t$.  It is an open question whether  the   pathwise uniqueness and strong existence hold for the matrix SDE  (\ref{eq:general}) itself.
   Note that the process $X_t$ takes values in the space ${\mathcal S}_p$ of dimension \\ $p(p+1)/2$ whereas the Brownian matrix in the SDE (\ref{eq:general})
   contains $p^2$ independent Brownian motions. Thus we have a   redundance phenomenon in the matrix SDE   (\ref{eq:general}). The SDEs system (\ref{eq:dlambda}) and  (\ref{eq:dh}) for $(\Lambda_t,H_t)$ has the advantage to be non-redundant, because it contains exactly
    $p(p+1)/2$ independent Brownian motions. 

  \end{remark}

%%%%%%%%%%%%%%%%%%%%%%%%%%%%%%%%%%%%%%%%%%%%%%%%%%%%%%%%%%%%%%%%%%%%%%%%%%%%%%%%%%%%%%%%%%%%%%%%%%%%%%%%%%%%%%%%%%%%%%%%%
%%%%%%%%%%%%%%%%%%%%%%%%%%%%%%%%%%%%%%%%%%%%%%%%%%%%%%%%%%%%%%%%%%%%%%%%%%%%%%%%%%%%%%%%%%%%%%%%%%%%%%%%%%%%%%%%%%%%%%%
%%%%%%%%%%%%%%%%%%%%%%%%%%%%%%%%%%%%%%%%%%%%%%%%%%%%%%%%%%%%%%%%%%%%%%%%%%%%%%%%%%%%%%%%%%%%%%%%%%%%%%%%%%%%%%%%%%%%%%
\section{Applications}\label{sec:appli}

\subsection{Noncolliding particle systems of squared Bessel processes }\label{sec:systemsBessel}

  In a recent paper by  Katori,Tanemura \cite{bib:katori2011},
particle systems of squared Bessel processes BESQ$^{(\nu)}$, $\nu>-1$, interacting with each other
by {\it long ranged repulsive forces} are studied.  If there are $N$ particles, their positions $X_i^{(\nu)}$
are given by  the following system of SDEs, see \cite{bib:katori2011} p.593:
\begin{eqnarray*}
 dX_i^{(\nu)}(t)&=&2\sqrt{X_i^{(\nu)}(t)} dB_i(t) + 2(\nu +1)dt + 4X_i^{(\nu)}(t)\sum_{j\not=i}\frac{dt} {X_i^{(\nu)}(t)-X_j^{(\nu)}(t)} \\
 &=&2\sqrt{X_i^{(\nu)}(t)} dB_i(t) + 2(\nu +N)dt +  2\sum_{j\not=i}\frac{X_i^{(\nu)}(t)+X_j^{(\nu)}(t) } {X_i^{(\nu)}(t)-X_j^{(\nu)}(t)}dt, \  i=1,\ldots, N
\end{eqnarray*}
with a collection of independent standard Brownian motions $\{B_i(t), i=1,\ldots, N\}$ and, if $-1<\nu<0$, with a reflection wall
at the origin. Theorem  \ref{th:eigenC} implies that  the processes $X_i^{(\nu)}(t)$
are the  eigenvalues of a complex Wishart (or Laguerre) process, with shape parameter $\delta=\nu +N$, see the end of the Section 3.2. It may be also seen
as a $\beta$-version of the real Wishart eigenvalue process, with $\beta=2$.

\begin{theorem}\label{th:PartSystem}
 The system of SDEs for a particle system of $N$ squared Bessel processes BESQ$^{(\nu)}$, with $0\le X_i^{(\nu)}(0) < X_2^{(\nu)}(0)<\ldots
 <X_N^{(\nu)}(0)$  admits a unique strong solution on $[0,\infty)$
 for $\nu \ge -1$.
\end{theorem}
\begin{proof}
Like for a  Squared Bessel process on $\R^+$, one must start with the following system of SDEs  
\begin{eqnarray*}
 dY_i^{(\nu)}(t)=2\sqrt{|Y_i^{(\nu)}(t)|} dB_i(t) + 2(\nu +N)dt +  2\sum_{j\not=i}\frac{|Y_i^{(\nu)}(t)|+|Y_j^{(\nu)}(t)| } {Y_i^{(\nu)}(t)-Y_j^{(\nu)}(t)}, \  i=1,\ldots, N,
\end{eqnarray*}
which is well defined on $\R^N$ up to the first collision time $\tau$. We suppose that $0\le Y_1^{(\nu)}(0)<Y_2^{(\nu)}(0)<\ldots<Y_N^{(\nu)}(0)$.
It follows from the  Corollary  \ref{cor:complex} that the collision time for the processes  $(Y_i^{(\nu)}(t)),\ i=1,\ldots, N$ is
 $\tau=\infty$ a.s. 
 
 First suppose that $\nu>-1$. 
 The Theorem \ref{thm:betterYamada} applied to the last system, with a standard use of localization
techniques as in the proof of Theorem  \ref{th:yamadamatrix}, gives the existence of a pathwise unique
strong solution $(Y_i^{(\nu)}(t))$. It remains to show that $Y_1^{(\nu)}(t)\ge 0$ for all $t>0$.

Denote 
$$
b_1(t)=\nu+N +\sum_{j\not=1}\frac{|Y_1^{(\nu)}(t)|+|Y_j^{(\nu)}(t)| } {Y_1^{(\nu)}(t)-Y_j^{(\nu)}(t)}
$$
We define two stopping times
\begin{eqnarray*}
\vartheta&=&\inf\{t>0|\ Y_1^{(\nu)}(t) <0\};\\
\kappa&=&\inf\{t>\vartheta|\ b_1(t)=0\}.
\end{eqnarray*}
Suppose that $\pr(\vartheta<\infty)>0$. Then there exists $T>0$ such that 
$\pr(\vartheta<T)>0$.
As $Y_1^{(\nu)}(\vartheta)=0$ and $b_1(\vartheta)=\nu +N-(N-1)=\nu+1>0$,
we see that if $\vartheta<\infty$ then  $\kappa>\vartheta$.

It follows from \cite{bib:ry99}Lemma 3.3, p.389
that the local time $L^0(Y_1^{(\nu)})=0$.  
Using Tanaka's formula \cite{bib:ry99}VI(1.2)  we obtain for $t\ge 0$
\begin{eqnarray*}
\ex(Y_1^{(\nu)}((\vartheta+t)\wedge \kappa\wedge T))^-=
- \ex\int_{\vartheta\wedge T}^{(\vartheta+t)\wedge \kappa\wedge T}
{\bf 1}_{\{Y_1^{(\nu)}(s)\le 0\}}dY_1^{(\nu)\}}(s)\\=
-2 \ex\int_{\vartheta\wedge T}^{(\vartheta+t)\wedge \kappa\wedge T}
{\bf 1}_{\{Y_1^{(\nu)}(s)\le 0\}}b_1(s)ds \le 0.
\end{eqnarray*}
In the last inequality we used the fact that $b_1(s)>0$ when $\vartheta\le s<\kappa$. Thus
$$Y_1^{(\nu)}((\vartheta+t)\wedge \kappa\wedge T) \ge 0$$
 for $t>0$ which contradicts the definition of
$\vartheta$. We deduce that $\vartheta=\infty$ almost surely.\\

In the case $\nu=-1$, let $T_0=\inf\{t>0|\ Y_1^{(\nu)}(t)=0\}$. Observe that
if $T_0<\infty$, then $b_1(T_0)=0$.
Define $\tilde Y_1^{(\nu)}(t)=Y_1^{(\nu)}(t)$ when $t<T_0$ and 
$\tilde Y_1^{(\nu)}(t)=0$ when $t\ge T_0$.
Then $(\tilde Y_1^{(\nu)}, Y_2^{(\nu)},\ldots,Y_N^{(\nu)})$ is a solution of
the same SDE system as $( Y_1^{(\nu)}, Y_2^{(\nu)},\ldots,Y_N^{(\nu)})$.
Consequently, by Theorem  \ref{thm:betterYamada}, we have $Y_1^{(\nu)}=\tilde Y_1^{(\nu)}\ge 0$.
\end{proof}
 %%%%%%%%%%%%%%%%%%%%%%%%%%%%%%%%%%%%%%%%%%%%%%%%%%%%%%%%%%%%%%%%%%%%%%%%%%%%%%%%%%  
   
\subsection{Wishart  stochastic differential equations} \label{sec:wishart}
Wishart processes on $\overline{{\mathcal S}_p^+}$ are matrix analogues of Squared Bessel processes on $\R^+$.
Wishart processes with shape parameter $n$ (which corresponds to the dimension of a BESQ on $\R^+$)  are simply constructed as $X_t=N_t^T\, N_t$
where $N_t$ is an $n\times p$ Brownian matrix. 
Let $\alpha>0$ and $B=(B_t)_{t\geq 0}$ be a Brownian $p\times p$ matrix. The Wishart stochastic differential equation for a Wishart process with
a shape parameter $\alpha$ is
\formula[eq:WSDE]{
\left\{
\begin{array}{rcl}
  dX_t &=& \sqrt{X_t}dB_t+dB_t^T \sqrt{X_t}+\alpha I dt\\
  X_0 &=& x_0\/.
\end{array}
\right.
}
It was introduced by Bru \cite{bib:b91} by first writing the SDE for  $X_t=N_t^T\, N_t$ and next replacing the parameter $n$ by $\alpha$.
It was shown in \cite{bib:b91} that if $x_0\in \overline{ {\mathcal S}^+_p}$ and $\alpha > p-1 $ then there exists a unique weak solution of (\ref{eq:WSDE}).
Also according to \cite{bib:b91}, the conditions $\alpha\geq p+1$
and $x_0\in   {\mathcal S}^+_p$ imply that (\ref{eq:WSDE}) has a unique strong solution. We reinforce considerably   these results.\\

Our methods apply to
  the following matrix stochastic differential equation
\formula[eq:GenWishart:SDE]{
   dY_t = \sqrt{|Y_t|}dB_t+dB^T_t\sqrt{|Y_t|}+\alpha Idt}
where $\alpha\in \R$,  $Y_0= y_0 \in \tilde S_p$  and $|Y_t|$ is defined by taking absolute values of eigenvalues of $Y_t$, see the beginning of Section 3. We have $g(x)=\sqrt{|x|}$, $h(x)=1$ and $G(x,y)=|x| + |y|$ for $x,y\in\R $. These functions satisfy the hypotheses of     Theorems \ref{prop:collision} and  \ref{th:yamadamatrix}.  

By Theorem \ref{th:eigen}, the eigenvalues of the  generalized Wishart process $Y_t$ verify
the following system of SDEs
$$
d\la_i=2 \sqrt {|\la_i|} d\nu_i +\left( \alpha     +\sum_{k\not=i} \frac{| \lambda_i|+|\lambda_k|
}
 {\lambda_i-\lambda_k}\right) dt. 
$$

First, using   Theorem    \ref{prop:collision} we obtain
\begin{corollary}
\label{cor:collisionWishart}
  For $\alpha \in \R$ and $ \lambda_1(0)<\lambda_2(0)<\ldots<\lambda_p(0) $ the eigenvalues $\lambda_i(t)$ never collide, i.e. the first collision time $\tau=\infty $ almost surely. 
\end{corollary}
Next, the  Theorem \ref{th:yamadamatrix} implies
\begin{corollary}
The SDEs system for the eigenvalues and eigenvectors $(\Lambda_t,H_t)$  corresponding to the generalized Wishart SDE (\ref{eq:GenWishart:SDE}) with $Y_0=y_0\in \tilde{\mathcal S}_p$  has a unique strong solution on $[0,\infty)$  for any $\alpha \in \R$.
\end{corollary}
In order to consider the equation  (\ref{eq:WSDE}), we must prove the non-negativity of the smallest eigenvalue
 of the process $Y_t$, when starting from a non-negative value. 
\begin{proposition}\label{prop:hit:zero}
If $\alpha\geq p-1$ and $\lambda_1(0)\ge 0$ then the process $\lambda_1(t)$ remains non-negative.
\end{proposition}
\begin{proof}
We  argue as in the proof of the Theorem \ref{th:PartSystem}.
\end{proof}
Consequently, using the unicity of  solutions in  Theorem \ref{th:yamadamatrix}, we obtain
\begin{corollary}
Consider the Wishart SDE (\ref{eq:WSDE}) with 
$x_0$ such that $0\leq \lambda_1(0)<\lambda_2(0)<\ldots<\lambda_p(0) $. Then the corresponding system of SDEs
for eigenvalue and eigenvector processes $(\Lambda_t,H_t)$ 
has a unique strong solution on $[0,\infty)$  for $\alpha \ge p-1$.
\end{corollary}

\begin{remark}
  Bru \cite{bib:b91} showed that for $\alpha>p-1$ the Wishart processes have the absolutely continuous Wishart laws
 which are very important  in multivariate statistics, see e.g. the monograph of Muirhead \cite{bib:Mu}. The 
 singular  Wishart  processes corresponding to $\alpha=1,\ldots p-1$  are obtained as $X_t=N_t^T\, N_t$
where $N_t$ is an $\alpha \times p$ Brownian matrix. Then $X_0=N_0^T\, N_0$ has eigenvalue 0 of multiplicity $p-\alpha$
 so $x_0\not\in  \tilde{\mathcal S}_p$.  
  \end{remark}

\begin{remark}
 An important perturbation of the Wishart SDE (\ref{eq:WSDE}) is the equation for the Wishart process with constant drift $c>0$,
 which may be also viewed as a squared matrix Ornstein-Uhlenbeck process
 \begin{equation}\label{eq:wishartDrift}
   dX_t = \sqrt{X_t}dB_t+dB_t^T \sqrt{X_t}+\alpha I dt + cX_tdt, \ \ X_0\in  \tilde{\mathcal S}_p.
 \end{equation}
This equation has the form (\ref{eq:general}) with $g(x)=\sqrt{x}, h(x)=1 $ and $b(x)=\alpha + cx$. By Theorems
 \ref{prop:collision} and \ref{th:yamadamatrix}, the SDEs system for its eigenvalue and eigenvector processes  has a unique  strong solution  with $t\in[0,\infty)$
 for  any $\alpha\ge p-1$, $c>0$ and   $0\leq \lambda_1(0)<\lambda_2(0)<\ldots<\lambda_p(0)$.
 More general  squared matrix Ornstein-Uhlenbeck processes were first studied by Bru \cite{bib:b91}
 and recently by Mayerhofer et al. \cite{bib:mayer}. Our spectral strong existence and uniqueness result for (\ref{eq:wishartDrift})
 is not covered by these papers.
 \end{remark}
 
 \begin{remark}
  The existence and pathwise unicity of strong solutions for the Wishart SDE (\ref{eq:WSDE}) for $\alpha\ge p-1$ remains an open problem.
  The difficulty of proving it  is   related to a redundance  in the SDE  (\ref{eq:WSDE}), cf. Remark \ref{dimensions}.
  On the other hand our result on the strong  existence and pathwise unicity of  eigenvalues and eigenvectors of $X_t$ 
  supports the conjecture of the  existence and pathwise unicity of strong solutions for the Wishart SDE (\ref{eq:WSDE}) for $\alpha\ge p-1$.
 \end{remark}

%%%%%%%%%%%%%%%%%%%%%%%%%%%%%%%%%%%%%%%%%%%%%%%%%%%%%%%%%%%%%%%%%%%%%%%%%%%
\subsection{  Matrix Jacobi processes}\label{sec:jacobi}

Let $0_p$ and $I_p$ be zero and identity $p\times p$ matrices. Define  $\mathcal{S}_p[0,I]=\{ X\in \mathcal{S}_p|\ 0_p\le X\le I_p\}$.
Denote by $\hat { \mathcal S}_p[0,I]=\{ X\in \mathcal{S}_p |\ 0_p< X< I_p\}$
and   by $ \tilde{\mathcal S}_p[0,I]$ the set of matrices in  $\mathcal{S}_p[0,I]$ with distinct eigenvalues.
 A matrix  Jacobi process of dimensions $(q,r)$, with $q\wedge r>p-1$, and with values in
 $\mathcal{S}_p[0,I]$, was defined and studied by Doumerc \cite{bib:doumerc} as a solution of the following
 matrix SDE, with respect to  a $p\times p$ Brownian matrix $B_t$
 \begin{equation} \label{eq:JSDE}
\left\{
\begin{array}{rcl}
  dX_t &=& \sqrt{X_t}dB_t \sqrt{I_p-X_t} + \sqrt{I_p-X_t}dB_t^T \sqrt{X_t} +(qI_m-(q+r)X_t)dt\\
  X_0 &=& x_0\in  {\mathcal S}_p[0,I]\/.
\end{array}
\right.
 \end{equation}
  In \cite{bib:doumerc} Th.9.3.1, p.135 it was shown that if $q\wedge r\ge p+1$  and
$x_0\in\hat { \mathcal S}_p[0,I]$  then (\ref{eq:JSDE}) has a unique   strong  solution in $\hat { \mathcal S}_p[0,I]$.
In the case  $q \wedge r\in (p-1,p+1)$ and $x_0\in  \tilde{\mathcal S}_p[0,I]$   the existence of a unique
solution in law was proved in  \cite{bib:doumerc}.  Our methods allow one to strengthen the results of Doumerc.
\begin{corollary}
 Let $q\wedge r\ge p-1$ and $x_0\in  \tilde{\mathcal S}_p[0,I]$. Then the SDEs system for the eigenvalue and eigenvector processes for the matrix SDE (\ref{eq:JSDE})
 has a unique strong solution for $t\in[0,\infty)$.
\end{corollary}
\begin{proof}
 We apply Theorems  \ref{prop:collision} and \ref{th:yamadamatrix}  with $g(x)=\sqrt{|x|}$, $h(x)=\sqrt{|1-x|}$
 and $b(x)=q-(q+r)x$. Next we prove similarly as in the proof of the Theorem \ref{th:PartSystem} that
 $0\le \lambda_1(t)<\ldots<\lambda_p(t)\le 1$.
\end{proof} 

 %%%%%%%%%%%%%%%%%%%%%%%%%%%%%%%%%%%%%%%%%%%%%%%%%%%%%%%%%%%%%%%%%%%%%%%%
\subsection{ $\beta$-Wishart  and
$\beta$-Jacobi   processes }\label{sec:systemsBeta}

Let $\beta>0$.  One calls a $\beta$-Wishart  process
 a solution of the system of SDEs
 
 \begin{equation}\label{eq:wishartBeta}
d\la_i=2 \sqrt {\la_i} d\nu_i +\beta\left( \alpha    +\sum_{k\not=i} \frac{ \lambda_i+\lambda_k
} {\lambda_i-\lambda_k}\right) dt.    
 \end{equation} 
The $\beta$-Wishart  processes were   studied  by Demni \cite{bib:demniARXIV}.
In the theory of random matrices and its physical applications, the  $\beta$-Wishart  processes are related to Chiral Gaussian Ensembles,
which were introduced
as effective
(approximation) theoretical models describing energy spectra of quantum particle
systems in high energy physics. Usually a symmetry of Hamiltonian is imposed and it fixes the value
of $\beta$ to be 1, 2 or 4, respectively  in real symmetric, hermitian and symplectic  cases.
 On the other hand, from the   point of view of statistical physics,
$\beta$ is regarded as the inverse temperature, $\beta=1/(k_B T)$,
and should be treated as a continuous positive parameter. In this sense, the
$\beta$-Wishart  systems are statistical mechanics models of ``$\log$-gases''
(The strength of the force between particles is proportional to the inverse of
distances. Then the potential, which is obtained by integrating the  force, is
logarithmic function  of the distance.
So the system is called a ``$\log$-gas'').  For more information on $\log$-gases,
see  the recent monograph of Forrester \cite{bib:forr}.
 
 In \cite{bib:demniARXIV} the existence and uniqueness of strong solutions of the SDE system
(\ref{eq:wishartBeta}) was established 
for $\beta >0 $  and $\alpha>p-1+\frac1\beta$.   Lépingle \cite{bib:lep} observed the last result
in the classical Wishart  case $\beta=1$.
Our  Theorem  \ref{thm:betterYamada}
 and  Corollary \ref{cor:beta}, together with comparison techniques like in the proof
 of the Theorem \ref{th:PartSystem}, imply the following result, not covered by results of Demni and  Lépingle.
 \begin{corollary}
  The SDE system (\ref{eq:wishartBeta}) with $0\leq \lambda_1(0)<\lambda_2(0)<\ldots<\lambda_p(0)$ has a unique strong solution for $t\in[0,\infty)$, 
  for any $\alpha\ge p-1$ and $\beta\ge 1$.
 \end{corollary}
%$\,$\\ 

 The $\beta$-Jacobi processes $(\lambda_i), i=1,\ldots,p$  are $[0,1]^p$-valued processes generalizing processes of eigenvalues
 of matrix Jacobi processes defined by (\ref{eq:JSDE}):
 \begin{equation}\label{eq:lambdaJacobiBeta}
  d\la_i=2 \sqrt {\la_i(1-\lambda_i)} d\nu_i +\beta\left( q-(q+r)\lambda_i   +\sum_{k\not=i} \frac{ \lambda_i(1-\lambda_k)+\lambda_k(1-\lambda_i)
} {\lambda_i-\lambda_k}\right) dt.    
 \end{equation}
 Indeed, for $\beta=1$ the formula (\ref{eq:lambdaJacobiBeta}) was shown in \cite{bib:doumerc} and it follows directly from the Theorem \ref{th:eigen}.
 $\beta$-Jacobi processes were recently studied by Demni in   \cite{bib:demniJACOBI}. He showed that the system (\ref{eq:lambdaJacobiBeta})
has a unique strong solution for all time $t$ when $\beta>0$  and $q\wedge r>p-1+1/\beta$. As an application of  Theorem  \ref{thm:betterYamada},
 Proposition \ref{prop:collision} and  the comparison techniques like in the proof
 of the Theorem \ref{th:PartSystem},   we improve this result when $\beta\ge 1$.
 \begin{corollary}
  The SDE system (\ref{eq:lambdaJacobiBeta}) with $0\leq \lambda_1(0)<\lambda_2(0)<\ldots<\lambda_p(0)\le 1$ has a unique strong solution for $t\in[0,\infty)$, 
  for any  $\beta\ge 1$ and  $q\wedge r\ge p-1$.
 \end{corollary}

%%%%%%%%%%%%%%%%%%%%%%%%%%%%%%%%%%%%%%%%%%%%%%%%%%%%%%%%%%%%%%%%%%%%%%%%%%%%%%%%%%%%%%%%%%%%%%%%%%%%%%%%%%%%%%%%%%%%%%%%%%%%

%%%%%%%%%%%%%%%%%%%%%%%%%%%%%%%%%%%%%%%%%%%%%%%%%%%%%%%%%%%%%%%%%%%%%%%%%%%%%%%%%%%%%%%%%%%%%

%{\bf  Wishart-Laguerre processes}.
%They are considered by K\"onig-O'Connell\cite{bib:konig} and Demni \cite{bib:demni}.
%It is proved in \cite{bib:demni} that in the case of the complex Wishart SDE:
%$$
%dX_t=\sqrt{X_t} dB_t + dB^*_t \sqrt{X_t} + 2\delta I_m dt
%$$
%where $B_t$ is a complex Brownian $p\times p$ matrix,
%a unique 
%strong solution exists for $\delta>p$ and a unique 
%weak solution exists for $\delta>p-1$.

%%%%%%%%%%%%%%%%%%%%%%%%%%%%%%%%%%%%%%%%%%%%%%%%%%%%%%%%%%%%%%%%%%%%%%%%%%%%%%%%%%

%%%%%%%%%%%%%%%%%%%%%%%%%%%%%%%%%%%%%%%%%%%%%%%%%%%%%%%%%%%%%%%%%%%%%%%%%%%%%%%%%

%%%%%%%%%%%%%%%%%%%%%%%%%%%%%%%%%%%%%%%%%%%%%%%%%%%%%%%%%%%%%%%%%%%%%%%%%%%%%%%%%
\begin{remark}
It would be interesting to extend our generalization of the Yamada-Watanabe theorem to the SDE's considered
by C\'epa-L\'epingle \cite{bib:cepa}. 
On the other hand the Jacobi eigenvalues processes being an  important example of the radial Cherednik processes,
we conjecture that the strong existence and unicity would hold for radial   Cherednik processes.
For radial Dunkl processes this is proved by Demni \cite{bib:demniCRAS}, using \cite{bib:cepa}.  
 \end{remark}
 {\bf Acknowledgements.} We thank  N. Demni, C. Donati,  M. Katori and  M. Yor  for discussions
 and bibliographical indications.

\end{document}